\documentclass[11pt,reqno]{amsart}

\usepackage{amsmath,amssymb,amsfonts}
\usepackage{amsthm}
\usepackage{mathrsfs}
\usepackage[T1]{fontenc}
\usepackage[a4paper,margin=1in]{geometry}
\usepackage{tikz}
\usetikzlibrary{positioning}

\theoremstyle{plain}
\newtheorem{theorem}{Theorem}[section]
\newtheorem{lemma}[theorem]{Lemma}
\newtheorem{proposition}[theorem]{Proposition}

\theoremstyle{remark}
\newtheorem{remark}[theorem]{Remark}

\theoremstyle{definition}
\newtheorem{definition}[theorem]{Definition}

\DeclareMathOperator{\Aut}{Aut}
\DeclareMathOperator{\Sym}{Sym}
\DeclareMathOperator{\Stab}{Stab}

\usepackage{hyperref}

\begin{document}

\title[The minimum size of a poset realizing $\mathbb{Z}_2 \times \mathbb{Z}_4$]{The Minimum Size of a Poset Realizing $\mathbb{Z}_2 \times \mathbb{Z}_4$ as its Automorphism Group}

\author{Ponaki Das}
\address{Department of Mathematics, North Eastern Hill University, Shillong 793022, India}
\email{ponaki.das20@gmail.com}

\author{Sainkupar Marwein Mawiong}
\address{Department of Basic Science and Social Science, North Eastern Hill University, Shillong 793022, India}
\email{skupar@gmail.com}



\begin{abstract}
For a finite group $G$, let $\beta(G)$ denote the minimum cardinality $|P|$ among finite posets $P$ whose automorphism group $\Aut(P)$ is isomorphic to $G$. While every finite group is realizable as the automorphism group of some finite poset, exact values of $\beta(G)$ are known only in special cases, most notably for cyclic groups. In this paper we prove that $\beta(\mathbb{Z}_2 \times \mathbb{Z}_4) = 14$; in particular, the product bound $\beta(G \times H) \le \beta(G) + \beta(H)$ is sharp in this case. The upper bound is realized by an explicit $14$-element poset $P_{14}$, whose automorphism group is computed by a height-function argument together with a rigidity analysis of its covering relations. The lower bound, which constitutes the substantive part of the proof, is established by a case analysis of the orbit decompositions of a hypothetical poset on at most $13$ points under a faithful $G$-action, organized according to the largest orbit size; in each case we construct an order-automorphism outside the given copy of $G$, contradicting $\Aut(P) \cong G$. Among non-cyclic groups, to our knowledge this is the first exact determination of $\beta(G)$ whose lower bound requires a structural analysis of this kind: for the other non-cyclic abelian groups of order at most $8$, namely $\mathbb{Z}_2 \times \mathbb{Z}_2$ and $\mathbb{Z}_2^3$, the value of $\beta$ is elementary. The arguments are closely adapted to the subgroup lattice of $\mathbb{Z}_2 \times \mathbb{Z}_4$.
\end{abstract}

\maketitle

\section{Introduction}\label{sec:intro}

A classical problem in algebraic combinatorics is to determine how finite groups arise as automorphism groups of combinatorial structures. For graphs, Frucht's theorem~\cite{Frucht1939} establishes that every finite group is isomorphic to the automorphism group of some finite undirected graph. The analogous question for finite posets was settled affirmatively by Birkhoff~\cite{Birkhoff1946}, who realized any finite group of order $n$ as the automorphism group of a poset on $n^2 + n$ elements; Thornton~\cite{Thornton1972} subsequently reduced the bound to $n(2r+1)$ when $G$ is $r$-generated (working in the equivalent setting of finite $T_0$-spaces), and the bound $n(r+2)$ was obtained by Frucht~\cite{Frucht1950} and, independently, by Barmak and Minian~\cite{BarmakMinian2009}. Babai~\cite{Babai1980} removed the dependence on the generator count, showing that $3n$ points always suffice; Barmak~\cite{Barmak2023} later gave a short, self-contained construction attaining the bound $4n$.

Once realizability is established, a natural quantitative refinement is to ask for the smallest realizing structure. Throughout this paper all posets are assumed to be finite and \emph{non-empty}; without the latter convention the empty poset would give the value $0$ for the trivial group, whereas with it the trivial group is realized minimally by a single point. (The convention is also used in the proof of Proposition~\ref{prop:product-bound}.) For a finite group $G$, define
\[
\beta(G) := \min\{\,|P| : P \text{ is a finite poset with } \Aut(P) \cong G\,\}.
\]
The set on the right is non-empty by Birkhoff's theorem~\cite{Birkhoff1946}, so $\beta(G)$ is a well-defined positive integer. Computing $\beta(G)$ amounts to proving matching upper and lower bounds: an explicit poset realizing $G$ with $\beta(G)$ elements, and a structural argument excluding all smaller posets. While realizability is settled in considerable generality, exact values of $\beta(G)$ are known only in special cases, most notably for cyclic groups: Barmak and Barreto~\cite{BarmakBarreto2024} determined $\beta(\mathbb{Z}_n)$ for every $n \ge 1$ (see Proposition~\ref{prop:cyclic} below for the precise statement), obtaining in particular
\[
\beta(\mathbb{Z}_2) = 2 \quad \text{and} \quad \beta(\mathbb{Z}_4) = 12.
\]
The same values of $\beta(\mathbb{Z}_n)$ were obtained independently by Gyenizse, Hajnal, and Z\'adori~\cite{GyenizseHajnalZadori2024}, who in addition gave height-one poset representations with at most four orbits for every finite group.
In addition, the thesis of Barreto~\cite{Barreto2021} computes $\beta(G)$ for finite cyclic groups and for finite abelian $p$-groups with $p \ge 11$ (see the discussion in \cite[\S1]{BarmakBarreto2024}); the general abelian case remains open.

For direct products, the general bound
\[
\beta(G \times H) \le \beta(G) + \beta(H)
\]
(Proposition~\ref{prop:product-bound}) yields $\beta(\mathbb{Z}_2 \times \mathbb{Z}_4) \le 14$. Establishing equality, however, requires a separate lower-bound argument; indeed, the product bound is not sharp in general, even for cyclic factors: $\beta(\mathbb{Z}_{12}) = \beta(\mathbb{Z}_3) + \beta(\mathbb{Z}_4) - 1$ by~\cite{BarmakBarreto2024}.

The main result of this paper is that the product bound is sharp for $\mathbb{Z}_2 \times \mathbb{Z}_4$:

\begin{theorem}\label{thm:intro-main}
$\beta(\mathbb{Z}_2 \times \mathbb{Z}_4) = 14$.
\end{theorem}

To place Theorem~\ref{thm:intro-main} in context, we record the values of $\beta$ for the remaining non-cyclic abelian groups of order at most $8$, namely $\mathbb{Z}_2 \times \mathbb{Z}_2$ and $\mathbb{Z}_2^3$. In contrast to Theorem~\ref{thm:intro-main}, both values follow from a short counting argument.

\begin{proposition}\label{prop:elementary}
$\beta(\mathbb{Z}_2 \times \mathbb{Z}_2) = 4$ and $\beta(\mathbb{Z}_2^3) = 6$.
\end{proposition}

\begin{proof}
For the upper bounds, the ordinal sums of two, respectively three, $2$-element antichains have automorphism groups $S_2 \times S_2 \cong \mathbb{Z}_2 \times \mathbb{Z}_2$ and $S_2 \times S_2 \times S_2 \cong \mathbb{Z}_2^3$, with $4$ and $6$ points; this is (iterated) Proposition~\ref{prop:product-bound} together with $\beta(\mathbb{Z}_2) = 2$. (Alternative height-one realizers of the same sizes are given by Gyenizse, Hajnal, and Z\'adori~\cite[Cor.~2.2]{GyenizseHajnalZadori2024}: the $4$-element crown for $\mathbb{Z}_2 \times \mathbb{Z}_2$, and its disjoint sum with a $2$-element antichain for $\mathbb{Z}_2^3$.)

For the lower bounds, suppose $\Aut(P) \cong \mathbb{Z}_2^k$ with $k \in \{2,3\}$, so that $G = \mathbb{Z}_2^k$ acts faithfully on $P$ and $|\Aut(P)| = 2^k$ divides $|P|!$. If $k = 2$ and $|P| \le 3$, then $|P|! \in \{1, 2, 6\}$ is not divisible by $4 = |G|$, a contradiction; hence $\beta(\mathbb{Z}_2 \times \mathbb{Z}_2) \ge 4$. For $k = 3$ the divisibility bound only gives $|P| \ge 4$ (as $8 \mid 4!$), so we argue through orbits. Suppose $|P| \le 5$. Each $G$-orbit $E$ has $|E| \in \{1,2,4\}$ (dividing $|G| = 8$); since $G$ is abelian, all points of $E$ have the same stabilizer $K_E$, which is the kernel of the action $G \to \Sym(E)$ and has order $|G|/|E| = 8/|E|$. Faithfulness means $\bigcap_E K_E = \{e\}$, the intersection over all orbits. We use the elementary bound $|H_1 \cap H_2| \ge |H_1|\,|H_2|/|G|$ for subgroups of $G$: since $G$ is abelian, $H_1 H_2$ is a subgroup of $G$ and so $|H_1 H_2| \le |G|$, while the product formula $|H_1 H_2| = |H_1|\,|H_2|/|H_1 \cap H_2|$ then yields the bound. Writing $|P| \le 5$ as a sum of orbit sizes from $\{1,2,4\}$, the parts exceeding $1$ are each at least $2$ and sum to at most $5$, so at most two such parts can appear (since three size-$2$ orbits would force $|P| \ge 6$). There are three cases. If some orbit has size $4$, then $|K_E| = 2$ and any remaining point is fixed (with stabilizer $G$), so $\bigcap_E K_E = K_E \ne \{e\}$. If no orbit has size $4$ but some orbit has size $2$, then at most two orbits have size $2$: a single one gives $\bigcap_E K_E \supseteq K_E$ of order $4$, while two of them, say with kernels $K_1, K_2$ of order $4$, give $|K_1 \cap K_2| \ge 4 \cdot 4 / 8 = 2$, so $\bigcap_E K_E$ has order at least $2$. If no orbit has size $2$ or $4$, then $P$ is a union of fixed points and $G$ acts trivially. In every case $\bigcap_E K_E \ne \{e\}$, contradicting faithfulness; hence $\beta(\mathbb{Z}_2^3) \ge 6$.
\end{proof}

By contrast, the lower bound $\beta(\mathbb{Z}_2 \times \mathbb{Z}_4) \ge 14$ does not, to our knowledge, admit a comparably short argument: it is established here through a structural analysis of all possible orbit decompositions of a hypothetical realizer on at most $13$ points. In this sense, to our knowledge, Theorem~\ref{thm:intro-main} is the first exact determination of $\beta(G)$ for a non-cyclic group in which the lower bound requires a non-trivial structural analysis. We emphasize that the result is not contained in~\cite{Barreto2021}, which treats cyclic groups and abelian $p$-groups with $p \ge 11$ only; nor in~\cite{GyenizseHajnalZadori2024}, whose exact-minimum results concern cyclic groups, the realizations of other groups there (including the non-cyclic $\mathbb{Z}_3^2$ and the quaternion group) being constructions that bound $\beta$ from above without determining it. Consistent with this, the exceptional constants at the prime powers $2,3,4,5,7$ in Proposition~\ref{prop:cyclic} already indicate that it is precisely the small primes that resist the general techniques.

\subsection*{Notation and conventions}

Throughout, $G = \mathbb{Z}_2 \times \mathbb{Z}_4$ with fixed generators $\sigma$ (of order $2$) and $\rho$ (of order $4$). The elements of $G$ are partitioned by order as follows:
\[
\begin{aligned}
& e \quad &&(\text{order } 1), \\
& \sigma,\ \rho^2,\ \sigma\rho^2 \quad &&(\text{order } 2), \\
& \rho,\ \rho^3,\ \sigma\rho,\ \sigma\rho^3 \quad &&(\text{order } 4).
\end{aligned}
\]
The proper non-trivial subgroups of $G$ are listed in Table~\ref{tab:subgroups}; in particular, $G$ has \emph{three} subgroups of order $4$ (two cyclic and one Klein four-group), a structural feature that underlies the subsequent case analysis. Two facts about this subgroup lattice are used repeatedly and recorded here for reference:
\begin{enumerate}
\item[(K1)]\label{K1} the three order-$2$ subgroups $\langle\sigma\rangle$, $\langle\rho^2\rangle$, $\langle\sigma\rho^2\rangle$ intersect pairwise trivially;
\item[(K2)]\label{K2} every order-$4$ subgroup of $G$ contains $\langle\rho^2\rangle$.
\end{enumerate}

\begin{table}[h]
\centering
\renewcommand{\arraystretch}{1.2}
\begin{tabular}{|c|l|l|l|}
\hline
\textbf{Order} & \textbf{Subgroup} & \textbf{Elements} & \textbf{Type} \\ \hline
$2$ & $\langle \sigma \rangle$ & $\{e, \sigma\}$ & \\
$2$ & $\langle \rho^2 \rangle$ & $\{e, \rho^2\}$ & \\
$2$ & $\langle \sigma\rho^2 \rangle$ & $\{e, \sigma\rho^2\}$ & \\ \hline
$4$ & $\langle \rho \rangle$ & $\{e, \rho, \rho^2, \rho^3\}$ & cyclic \\
$4$ & $\langle \sigma\rho \rangle$ & $\{e, \sigma\rho, \rho^2, \sigma\rho^3\}$ & cyclic \\
$4$ & $\langle \sigma, \rho^2 \rangle$ & $\{e, \sigma, \rho^2, \sigma\rho^2\}$ & Klein \\ \hline
\end{tabular}
\caption{Proper non-trivial subgroups of $G = \mathbb{Z}_2 \times \mathbb{Z}_4$, with properties~\hyperref[K1]{(K1)} and~\hyperref[K2]{(K2)} noted above. The ``Type'' column records the cyclic/Klein classification of the order-$4$ subgroups used in Definition~\ref{def:canonical-partition}; it is left blank for the order-$2$ subgroups, to which it does not apply.}
\label{tab:subgroups}
\end{table}

By the orbit--stabilizer theorem, every $G$-orbit in $P$ has size dividing $|G| = 8$, hence size in $\{1, 2, 4, 8\}$. We classify $G$-actions on $P$ into four types according to the largest orbit size:
\emph{Type~A}: at least one orbit of size $8$;
\emph{Type~B}: no orbit of size $8$ and exactly one orbit of size $4$;
\emph{Type~C}: no orbit of size $8$ and at least two orbits of size $4$;
\emph{Type~D}: all orbit sizes lie in $\{1,2\}$.
These types are mutually exclusive and exhaustive: an action either has an orbit of size $8$ (Type~A) or it does not, and in the latter case the largest orbit size lies in $\{1,2,4\}$; if it is at most $2$ the action is of Type~D, while if it equals $4$ the action is of Type~B or~C according as the number of $4$-orbits is exactly one or at least two.

\subsection*{Methodology}

The upper bound is achieved by an explicit $14$-element poset $P_{14}$ (Definition~\ref{def:p14}), for which we verify that $\Aut(P_{14}) \cong \mathbb{Z}_2 \times \mathbb{Z}_4$ via a height-function argument and rigidity of the covering relations.

The lower bound constitutes the substantive content. Suppose for contradiction that $P$ is a finite poset with $|P| \le 13$ and $\Aut(P) \cong \mathbb{Z}_2 \times \mathbb{Z}_4$. Fixing such an isomorphism and composing with it, we obtain a faithful action of $G = \mathbb{Z}_2 \times \mathbb{Z}_4$ on $P$ by order-automorphisms; this action realizes $G$, which has order $8$, as a subgroup of $\Aut(P) \subseteq \Sym(P)$. Throughout the lower-bound analysis we treat $G$ as this fixed subgroup of $\Aut(P)$, and we write $\Stab_G(x)$, orbits, and kernels with respect to this action.

We classify the $G$-orbit decomposition of $P$ according to the four types defined above. In each type we construct an order-automorphism $\phi \in \Aut(P)$ with $\phi \notin G$ (equivalently, $\Aut(P) \supsetneq G$ as subgroups of $\Sym(P)$). The underlying strategy, permuting the elements of a single orbit while fixing the rest and verifying that the result extends to an automorphism, follows the orbit-wise extension technique of Barmak and Barreto~\cite{BarmakBarreto2024} (see also~\cite{Barmak2023}). Such a $\phi$ forces $|\Aut(P)| \ge 9 > 8 = |G|$, so $\Aut(P)$ cannot be isomorphic to $G$, contradicting the assumption. The most delicate type is~C (no $8$-orbit and at least two $4$-orbits), which we further split via a structural dichotomy: either a uniquely-kerneled $4$-orbit exists (handled by Theorem~\ref{thm:strategy1}, which further sub-divides according to whether the uniquely-kerneled orbit is of cyclic or Klein type, the Klein case requiring the configuration analysis of Lemma~\ref{lem:klein-unique-config}), or exactly two $4$-orbits exist and share a common kernel (handled by Theorem~\ref{thm:strategy2}). The latter is treated by a parity-and-adjacency analysis of the comparability pattern between the two $4$-orbits.

\subsection*{Organization}

Section~\ref{sec:prelim} establishes basic equivariance lemmas, the canonical partition of a $4$-orbit, and the stabilizer enumeration (Lemma~\ref{lem:stab-enum}) used throughout. Section~\ref{sec:smalltypes} disposes of Types~D, A, and~B, and Section~\ref{sec:typeC} carries out the Type~C analysis. Section~\ref{sec:lower} assembles the lower bound, and Section~\ref{sec:construction} presents the construction $P_{14}$, determines its automorphism group, and completes the proof of Theorem~\ref{thm:intro-main}. Section~\ref{sec:remarks} concludes with open problems.

\section{Preliminaries}\label{sec:prelim}

For elements $x, y$ of a poset $P$, we write $x \prec y$ (read ``$y$ covers $x$'') to mean $x < y$ and there is no $z$ with $x < z < y$.

For an element $x$ of a finite poset $P$, the \emph{height} $h(x)$ is the length (number of covering edges) of a longest chain whose top element is $x$, and the \emph{height of $P$} is $\mathrm{ht}(P) := \max_{x \in P} h(x)$. Since an order-isomorphism carries chains to chains of equal length, every automorphism of $P$ preserves $h$, and therefore preserves each height level $h^{-1}(\ell)$ setwise.

\begin{proposition}\label{prop:product-bound}
For finite groups $G$ and $H$,
\[
\beta(G \times H) \le \beta(G) + \beta(H).
\]
\end{proposition}

\begin{proof}
The construction is the standard ordinal sum, used in this context by Barmak~\cite{Barmak2023}; we sketch it. Let $P_G$ and $P_H$ be minimal posets realizing $G$ and $H$, and form the \emph{ordinal sum} $P_G \oplus P_H$: the disjoint union $P_G \sqcup P_H$ with the orders of the two summands, augmented by $u < v$ for every $u \in P_G$, $v \in P_H$. Heights separate the summands: a chain ending in $P_G$ lies in $P_G$, so its top has height at most $\mathrm{ht}(P_G)$, whereas a chain ending at $y \in P_H$ may be prepended by a longest chain of $P_G$ (here $P_G \ne \emptyset$ by our standing non-emptiness convention), so $h(y) \ge \mathrm{ht}(P_G) + 1$. Since automorphisms preserve heights, they preserve the partition $\{P_G, P_H\}$ and restrict to automorphisms of each summand; conversely, every pair in $\Aut(P_G) \times \Aut(P_H)$ defines an automorphism of $P_G \oplus P_H$. Hence $\Aut(P_G \oplus P_H) \cong G \times H$, and $\beta(G \times H) \le |P_G| + |P_H| = \beta(G) + \beta(H)$.
\end{proof}

Since the entire upper bound of Theorem~\ref{thm:intro-main} rests on the next citation, we quote it in full.

\begin{proposition}[{\cite[Cor.~4.2]{BarmakBarreto2024}}]\label{prop:cyclic}
Let $n = p_1^{r_1} p_2^{r_2} \cdots p_k^{r_k}$, where the $p_i$ are pairwise distinct primes and $r_i \ge 1$. Then
\[
\beta(\mathbb{Z}_n) \;=\; \sum_{i=1}^{k} b(p_i^{r_i})\, p_i^{r_i} \;-\; \varepsilon(n),
\]
where $\varepsilon(n) = 1$ if $3$ and $4$ both exactly divide $n$, and $\varepsilon(n) = 0$ otherwise; here $b(2) = 1$, $b(3) = b(4) = b(5) = b(7) = 3$, and $b(p^r) = 2$ for every other prime power. In particular,
\[
\beta(\mathbb{Z}_2) = 2 \cdot b(2) = 2 \quad\text{and}\quad \beta(\mathbb{Z}_4) = 4 \cdot b(4) = 12,
\]
and only these two values are used in this paper.
\end{proposition}

Applying Proposition~\ref{prop:product-bound} together with Proposition~\ref{prop:cyclic}, we obtain
\[
\beta(\mathbb{Z}_2 \times \mathbb{Z}_4)
\le \beta(\mathbb{Z}_2) + \beta(\mathbb{Z}_4)
= 2 + 12
= 14.
\]
This yields the upper bound in Theorem~\ref{thm:intro-main}.

\begin{lemma}\label{lem:antichain}
If a finite group acts on a finite poset by order-preserving bijections, every orbit is an antichain.
\end{lemma}

\begin{proof}
Suppose $x < y$ with $x,y$ in the same orbit, so $y = g \cdot x$ for some $g \in G$; in particular $g \cdot x = y \neq x$. Since the action is order-preserving, applying $g$ to the relation $x < g \cdot x$ repeatedly gives $g^i \cdot x < g^{i+1} \cdot x$ for every $i \ge 0$. Let $k \ge 1$ be minimal with $g^k \cdot x = x$ (it exists since $G$ is finite). Chaining the inequalities for $i = 0, 1, \dots, k-1$ yields, by transitivity of $<$,
\[
x < g^{k} \cdot x = x,
\]
i.e., $x < x$, a contradiction since $<$ is irreflexive.
\end{proof}

\noindent\textbf{Notation.} For an element $a$ of a poset $P$ and a subset $B \subseteq P$, define the upper and lower neighborhoods of $a$ in $B$ by
\[
U_a^B := \{b \in B : a < b\}, \qquad D_a^B := \{b \in B : b < a\}.
\]
We will mostly apply this when $a$ lies in a $G$-orbit $A$ and $B$ is a different $G$-orbit.

\begin{lemma}\label{lem:equivariance}
Let $G$ act on a finite poset $P$ by order-preserving bijections, and let $a \in P$ and $B \subseteq P$. Then the sets $U_a^B$, $D_a^B$ are $\Stab_G(a)$-invariant whenever $B$ is $G$-invariant. Moreover, for every $g \in G$ and every $G$-invariant $B$, $U_{ga}^B = g \cdot U_a^B$ and $D_{ga}^B = g \cdot D_a^B$.
\end{lemma}

\begin{proof}
Let $h \in \Stab_G(a)$ and let $B$ be $G$-invariant. If $b \in U_a^B$, then $a < b$ and $b \in B$, so $a = ha < hb$ since the action is order-preserving and $hb \in B$ by $G$-invariance; thus $h \cdot U_a^B \subseteq U_a^B$, with equality since $h^{-1} \in \Stab_G(a)$ gives the reverse inclusion. The argument for $D_a^B$ is identical with the inequality reversed. For equivariance, $b \in U_{ga}^B$ iff $b \in B$ and $ga < b$ iff $g^{-1}b \in B$ and $a < g^{-1}b$ (using $G$-invariance of $B$) iff $g^{-1}b \in U_a^B$ iff $b \in g \cdot U_a^B$; and analogously for $D$.
\end{proof}

\begin{lemma}\label{lem:S4-embed}
$G$ does not embed into $S_n$ for $n \le 4$.
\end{lemma}

\begin{proof}
For $n \le 3$ we have $|S_n| < 8 = |G|$, so no injection $G \hookrightarrow S_n$ exists.

For $n = 4$, suppose for contradiction that $f \colon G \hookrightarrow S_4$ is an injective homomorphism. The image $f(\rho) \in S_4$ has order $4$ (since $\rho$ has order $4$ and $f$ is injective), so $f(\rho)$ is a $4$-cycle in $S_4$. Denote it $\pi := f(\rho)$. We claim $C_{S_4}(\pi) = \langle \pi \rangle$. The inclusion $\langle \pi \rangle \subseteq C_{S_4}(\pi)$ is immediate. For the reverse, write $\pi = (1\;2\;3\;4)$; any $\tau \in S_4$ satisfying $\tau\pi = \pi\tau$ satisfies $\tau(\pi(x)) = \pi(\tau(x))$ for all $x$. The condition $\tau(2) = \tau(\pi(1)) = \pi(\tau(1))$ then determines $\tau(2)$ from $\tau(1)$, and similarly $\tau(3)$ and $\tau(4)$, so $\tau$ is completely determined by the single value $\tau(1) \in \{1,2,3,4\}$. The four resulting permutations are exactly $\pi^0, \pi^1, \pi^2, \pi^3$, giving $C_{S_4}(\pi) \subseteq \langle\pi\rangle$ and hence equality.

Since $\sigma$ commutes with $\rho$ in $G$, the image $f(\sigma)$ commutes with $\pi$ in $S_4$; hence $f(\sigma) \in C_{S_4}(\pi) = \langle \pi \rangle$. The unique element of order $2$ in $\langle \pi \rangle$ is $\pi^2$, so $f(\sigma) = \pi^2 = f(\rho)^2 = f(\rho^2)$. Injectivity of $f$ forces $\sigma = \rho^2$ in $G$, contradicting the fact that $\sigma$ and $\rho^2$ are distinct elements of order $2$ in $\mathbb{Z}_2 \times \mathbb{Z}_4$.
\end{proof}

\subsection{4-orbit types, canonical partition, and the stabilizer enumeration}\label{subsec:fourorbits}

Since $G$ is abelian, all element stabilizers within a single $G$-orbit coincide: for $x$ in an orbit $E$ and any $g \in G$, one has $\Stab_G(g \cdot x) = g\, \Stab_G(x)\, g^{-1} = \Stab_G(x)$. For any orbit $E$, write $K_E := \Stab_G(x)$ for any $x \in E$. Equivalently, $K_E = \bigcap_{x \in E} \Stab_G(x)$ is the kernel of the restricted action $G \to \Sym(E)$; this is the meaning of the term \emph{kernel} throughout, in particular in the phrase ``cross-kernel structure'' of Section~\ref{sec:typeC}. By the orbit--stabilizer theorem, $|K_E| = |G|/|E|$, so $K_E$ has order $2$ when $E$ is a $4$-orbit.

For a $4$-orbit $C$, $K_C$ is an order-$2$ subgroup of $G$, and we distinguish two types. We say $C$ is of \emph{cyclic type} if $K_C \in \{\langle \sigma \rangle, \langle \sigma\rho^2 \rangle\}$; in this case $G/K_C \cong \mathbb{Z}_4$ and $\rho$ acts on $C$ as a $4$-cycle, and we label $C = \{c_0, c_1, c_2, c_3\}$ so that $\rho \cdot c_k = c_{k+1 \bmod 4}$. We say $C$ is of \emph{Klein type} if $K_C = \langle \rho^2 \rangle$; in this case $G/K_C \cong V_4$ acts regularly on $C$.

\begin{definition}\label{def:canonical-partition}
The \emph{canonical partition} $\{P_1^C, P_2^C\}$ of a $4$-orbit $C$ is defined as follows. If $C$ is of cyclic type, then $P_1^C = \{c_0, c_2\}$ and $P_2^C = \{c_1, c_3\}$ (the orbits of $\rho^2|_C$). If $C$ is of Klein type, then $\{P_1^C, P_2^C\}$ is the orbit partition of $C$ under the permutation induced by $\sigma$ on $C$. This is well defined. First, $\sigma$ acts on $C$ without fixed points: since $G$ is abelian, $K_C = \langle\rho^2\rangle$ is the stabilizer of \emph{every} element of $C$, so $\sigma \notin K_C$ fixes \emph{no} element of $C$; as $\sigma$ has order $2$, it induces a fixed-point-free involution on the $4$-element set $C$, whose orbit type is necessarily $2+2$, giving a genuine partition $\{P_1^C, P_2^C\}$ into two $2$-element classes. Second, $\sigma$ and $\sigma\rho^2$ induce the same permutation of $C$: their product $\sigma \cdot (\sigma\rho^2) = \rho^2$ lies in $K_C = \langle \rho^2 \rangle$, so they act identically on $C$, and the partition is independent of which of the two elements is used to define it.
\end{definition}

The following enumeration of the possible neighborhood sets $U_z^C, D_z^C$ along a cyclic-type $4$-orbit is used in the analyses of Types~B and~C alike. Note that, by Table~\ref{tab:subgroups}, the subgroups appearing in its hypothesis are exactly the possible stabilizers of elements lying in $G$-orbits of size $1$ or $2$, so the lemma applies to every element of $P$ whose orbit has size at most $2$.

\begin{lemma}\label{lem:stab-enum}
Let $G$ act on a finite poset $P$ by order-preserving bijections, let $C$ be a cyclic-type $4$-orbit, and let $z \in P$ be an element with
\[
\Stab_G(z) \in \{\,G,\ \langle \rho \rangle,\ \langle \sigma\rho \rangle,\ \langle \sigma, \rho^2 \rangle\,\}.
\]
Then
\[
U_z^C,\ D_z^C \in \{\emptyset, P_1^C, P_2^C, C\};
\]
moreover, if $\Stab_G(z) \ne \langle \sigma, \rho^2 \rangle$ then $U_z^C, D_z^C \in \{\emptyset, C\}$, so the partition values $P_1^C, P_2^C$ can occur only when $\Stab_G(z) = \langle \sigma, \rho^2 \rangle$.
\end{lemma}

\begin{proof}
By Lemma~\ref{lem:equivariance} (with $B = C$, which is $G$-invariant), both sets are $\Stab_G(z)$-invariant, so it suffices to determine the invariant subsets of $C$ under the image of $\Stab_G(z)$ in $\Sym(C)$.

If $\Stab_G(z) \in \{G, \langle\rho\rangle\}$, the image contains the image of $\rho$, which generates $\langle \rho|_C \rangle$ and acts transitively on $C$; the invariant subsets are $\emptyset$ and $C$.

If $\Stab_G(z) = \langle \sigma\rho \rangle$, the image is generated by the image of $\sigma\rho$ in $G/K_C \cong \mathbb{Z}_4$. When $K_C = \langle \sigma \rangle$ we have $\sigma \in K_C$, so the image of $\sigma\rho$ equals that of $\rho$; when $K_C = \langle \sigma\rho^2 \rangle$, the relation $\sigma\rho^2 \in K_C$ gives the image of $\sigma$ equal to that of $\rho^{-2} = \rho^2$, hence the image of $\sigma\rho$ equal to that of $\rho^3$. In either case the image of $\sigma\rho$ has order $4$ and generates $G/K_C$, so $\langle \sigma\rho \rangle$ acts transitively on $C$ and again the invariant subsets are $\emptyset$ and $C$.

If $\Stab_G(z) = \langle \sigma, \rho^2 \rangle$, then for either choice of $K_C \in \{\langle\sigma\rangle, \langle\sigma\rho^2\rangle\}$ the image in $\Sym(C)$ is $\{\mathrm{id}, \rho^2|_C\} = \{\mathrm{id}, (c_0c_2)(c_1c_3)\}$, whose invariant subsets are exactly $\emptyset$, $P_1^C$, $P_2^C$, and $C$.

These arguments use only $\Stab_G(z)$-invariance, not the direction of the order relation, so they apply equally to $U_z^C$ and $D_z^C$.
\end{proof}

\section{Types D, A, and B}\label{sec:smalltypes}

\subsection{Type D: orbits of size at most $2$}\label{sec:typeD}

\begin{lemma}\label{lem:typeD}
If all orbits of a $G$-action on $P$ have size $\le 2$, then $\rho^2$ acts trivially on $P$.
\end{lemma}

\begin{proof}
The set $P$ decomposes as a disjoint union of $G$-orbits. For each orbit $E$ with $|E| \le 2$, the restriction $\rho|_E$ is a permutation of $E$ of order dividing $|\Sym(E)| \le 2$; hence $(\rho|_E)^2 = \mathrm{id}_E$. Thus $\rho^2$ acts as the identity on each orbit, and therefore on all of $P$.
\end{proof}

In particular, $G$ does not act faithfully in Type~D.

\subsection{Type A: $8$-orbit configurations}\label{sec:8orbit}

\begin{lemma}\label{lem:8orbit}
Let $G$ act faithfully on a finite poset $P$ with $|P| \le 13$, and suppose $P$ contains an orbit $C$ of size $8$ together with at least one further orbit (so in particular $|P| \ge 9$). Then $\Aut(P) \supsetneq G$.
\end{lemma}

\begin{proof}
The orbit $C$ has size $8$, and $|G| = 8$. Hence for any $c \in C$,
\[
|\Stab_G(c)| = |G|/|C| = 1,
\]
so the stabilizer is trivial. Thus the action of $G$ on $C$ is free and transitive, i.e., regular, and $K_C = \{e\}$.

Set $H := \bigcap_{E \ne C} K_E$, the intersection over all $G$-orbits $E \ne C$. We show $H \ne \{e\}$ by enumerating the possibilities for these orbits, which together contain at most $|P| - 8 \le 5$ elements. In particular, since $5 < 8$, no further orbit of size $8$ can occur, and at most one $4$-orbit can occur (two would need $8$ elements). We split on the number of $4$-orbits.

\emph{No $4$-orbit.} Then every orbit $E \ne C$ has size $1$ or $2$, so $|K_E| \in \{8, 4\}$, i.e., $K_E$ is either $G$ or an order-$4$ subgroup. By property~\hyperref[K2]{(K2)}, every order-$4$ subgroup contains $\langle\rho^2\rangle$, and $G \supseteq \langle\rho^2\rangle$ as well; hence $K_E \supseteq \langle\rho^2\rangle$ for every such $E$, giving $H \supseteq \langle \rho^2 \rangle \ne \{e\}$.

\emph{Exactly one $4$-orbit $D$.} The elements outside $C \cup D$ number at most $13 - 8 - 4 = 1$, so the remaining orbits form either nothing or a single fixed point $F$, which has $K_F = G \supseteq K_D$. Hence $H = K_D$ in either sub-case (adjoining $K_F = G$ to an intersection changes nothing), and $|H| = |K_D| = 2$, so $H \ne \{e\}$.

In both cases $H \ne \{e\}$.

Choose $h \in H \setminus \{e\}$ and $c_0 \in C$, and let $c_1 = h \cdot c_0$. Since the action on $C$ is free and $h \ne e$, we have $c_1 \neq c_0$. Define $\phi$ to swap $c_0, c_1$ and fix every other element of $P$.

We verify $\phi \in \Aut(P)$. As $\phi$ is the identity off $\{c_0,c_1\}$, the only relations that could be affected are those between $\{c_0,c_1\}$ and the rest of $P$. Relations \emph{within} $C$ need no checking: $C$ is an antichain (Lemma~\ref{lem:antichain}), so in particular $c_0$ and $c_1$ are incomparable to each other and to every other element of $C$, and these (non-)relations are unchanged by a permutation internal to $C$. For each orbit $E \ne C$, the element $h \in K_E$ acts trivially on $E$, so by Lemma~\ref{lem:equivariance} (with $B = E$, which is $G$-invariant),
\[
U_{c_1}^E = U_{h\cdot c_0}^E = h \cdot U_{c_0}^E = U_{c_0}^E,
\]
and likewise $D_{c_1}^E = D_{c_0}^E$. Thus $c_0$ and $c_1$ have identical upper and lower neighborhoods in every $E \ne C$, so swapping them preserves all order relations between $C$ and $P \setminus C$; relations within $P \setminus C$ are preserved since $\phi$ fixes $P \setminus C$ pointwise. Hence $\phi \in \Aut(P)$.

Finally, $\phi \notin G$: every non-identity $g \in G$ acts on $C$ regularly, hence without fixed points, whereas $\phi|_C$ is a transposition fixing six elements of $C$. Therefore $\Aut(P) \supsetneq G$.
\end{proof}

\begin{lemma}\label{lem:pure-8orbit}
If $G$ acts on $P$ by order-automorphisms and $P$ is itself a single $G$-orbit of size $8$, then $\Aut(P) \cong S_8 \not\cong G$.
\end{lemma}

\begin{proof}
By Lemma~\ref{lem:antichain} every $G$-orbit is an antichain, so $P$ is an antichain. Every bijection on an antichain is order-preserving, hence $\Aut(P) = \Sym(P) \cong S_8$. Since $|S_8| = 40320 \neq 8 = |G|$, we have $\Aut(P) \not\cong G$.
\end{proof}

\subsection{Type B: single-$4$-orbit configurations}\label{sec:hiddeninv}

\begin{lemma}\label{lem:hiddeninv}
Let $G$ act faithfully on $P$ with exactly one orbit $C$ of size $4$, no orbit of size $8$, and all other orbits of size $\le 2$. Then there exists $\tau \in \Aut(P) \setminus G$.
\end{lemma}

\begin{proof}
\emph{Step 1: $C$ is of cyclic type.} If $K_C = \langle \rho^2 \rangle$, then $\rho^2$ fixes every element of $C$. Since $C$ is a $G$-orbit, $P \setminus C$ is $G$-invariant (a union of orbits), and the $G$-action restricts to $P \setminus C$ with all orbits of size $\le 2$ by hypothesis. Applying Lemma~\ref{lem:typeD} to this restricted action, $\rho^2$ acts trivially on $P \setminus C$ as well. Hence $\rho^2$ acts trivially on all of $P$, contradicting faithfulness. So $K_C \in \{\langle\sigma\rangle, \langle\sigma\rho^2\rangle\}$ and $C$ is cyclic type.

\emph{Step 2: Construction.} Set $\tau(c_1) = c_3$, $\tau(c_3) = c_1$, and $\tau(z) = z$ for all other $z \in P$ (so $\tau$ fixes $c_0, c_2$ and all of $P \setminus C$). The transposition $(c_0\, c_2)$ would serve equally well: each fixes the canonical partition $\{P_1^C, P_2^C\}$ setwise and lies outside the image of $G$ in $\Sym(C)$.

\emph{Step 3: Order-preservation.} Since $C$ is an antichain (Lemma~\ref{lem:antichain}), $\tau$ preserves all (non-)relations within $C$ and, being the identity there, within $P \setminus C$; only relations between $C$ and $P \setminus C$ require verification, and for these it suffices that $U_z^C$ and $D_z^C$ are $\tau$-invariant for every $z \in P \setminus C$. By hypothesis every $z \in P \setminus C$ lies in an orbit of size $1$ or $2$, so $\Stab_G(z)$ is the full group $G$ (if $z$ is a fixed point) or one of the three order-$4$ subgroups $\langle \rho \rangle$, $\langle \sigma\rho \rangle$, $\langle \sigma, \rho^2 \rangle$ (if $z$ lies in a $2$-orbit). Lemma~\ref{lem:stab-enum} therefore gives
\[
U_z^C,\ D_z^C \in \{\emptyset, P_1^C, P_2^C, C\}.
\]
The map $\tau$ fixes $P_1^C = \{c_0, c_2\}$ pointwise and exchanges the two elements of $P_2^C = \{c_1, c_3\}$, so each of the four candidate subsets is $\tau$-invariant. Hence $\tau \in \Aut(P)$.

\emph{Step 4: $\tau \notin G$.} The image of $G$ in $\Sym(C)$ is $\langle (c_0c_1c_2c_3) \rangle$, whose only non-identity involution is $(c_0c_2)(c_1c_3)$. The transposition $\tau|_C = (c_1c_3)$ is not in this image, so $\tau \notin G$.
\end{proof}

\section{Type C: multi-$4$-orbit configurations}\label{sec:typeC}

Throughout this section, let $m$ denote the number of $G$-orbits of size $4$ in $P$. In a Type~C configuration, $m \ge 2$. When $m = 2$ we write the two $4$-orbits as $C, D$, and when $m = 3$ we write them as $C_1, C_2, C_3$.

\subsection{Cross-kernel structure}

\begin{lemma}\label{lem:cross-kernel}
Let $C, C'$ be $4$-orbits in $P$ with $K_{C'} \ne K_C$. For every $c' \in C'$,
\[
U_{c'}^C,\ D_{c'}^C \in \{\emptyset, P_1^C, P_2^C, C\}.
\]
\end{lemma}

\begin{proof}
By Lemma~\ref{lem:equivariance} (with $B = C$, which is $G$-invariant), each set is $K_{C'}$-invariant. The non-identity element $h \in K_{C'}$ satisfies $h \notin K_C$: by property~\hyperref[K1]{(K1)}, distinct order-$2$ subgroups of $G$ intersect trivially, so $K_{C'} \cap K_C = \{e\}$, whence the non-identity $h \in K_{C'}$ is not in $K_C$. Thus $h$ acts non-trivially on $C$. Since $K_{C'} = \{e, h\}$, the $K_{C'}$-invariant subsets of $C$ are exactly the unions of $\langle h \rangle$-orbits on $C$. We split by the type of $C$.

\emph{Case 1: $C$ is cyclic type.} The image of $G$ in $\Sym(C)$ is the cyclic group $\langle \rho|_C \rangle = \langle (c_0c_1c_2c_3) \rangle$, which has a unique non-trivial involution, namely $\rho^2|_C$. As $h$ acts on $C$ as a non-trivial involution (it has order $2$ and acts non-trivially), its image must be $\rho^2|_C = (c_0c_2)(c_1c_3)$, with orbit partition $\{P_1^C, P_2^C\} = \{\{c_0, c_2\}, \{c_1, c_3\}\}$. The $\langle h \rangle$-invariant subsets of $C$ are therefore $\{\emptyset, P_1^C, P_2^C, C\}$.

\emph{Case 2: $C$ is Klein type.} Then $K_C = \langle \rho^2 \rangle$, and $G/K_C \cong V_4$ acts regularly on $C$. Write $\bar g := gK_C$, so
\[
V_4 = \{e, \bar{\sigma}, \bar{\rho}, \bar{\sigma\rho}\}, \qquad \bar{\sigma} = \overline{\sigma\rho^2} \text{ since } \sigma\cdot(\sigma\rho^2) = \rho^2 \in K_C.
\]
By Definition~\ref{def:canonical-partition}, $\{P_1^C, P_2^C\}$ is the orbit partition of $C$ under $\bar{\sigma}$. By Table~\ref{tab:subgroups} and property~\hyperref[K2]{(K2)}, the only order-$2$ subgroup that can serve as the kernel of a Klein-type orbit is $\langle \rho^2 \rangle$; a Klein-type $C'$ would therefore have $K_{C'} = \langle \rho^2 \rangle = K_C$, contradicting $K_{C'} \ne K_C$. Hence $C'$ is cyclic-type and $K_{C'} \in \{\langle \sigma \rangle, \langle \sigma\rho^2 \rangle\}$; the non-identity element of either kernel ($\sigma$ or $\sigma\rho^2$) projects to $\bar{\sigma}$. Hence $h$ maps to $\bar\sigma$, the orbit partition of $C$ under $\langle h \rangle$ is exactly $\{P_1^C, P_2^C\}$, and the $\langle h \rangle$-invariant subsets are $\{\emptyset, P_1^C, P_2^C, C\}$.
\end{proof}

\begin{remark}\label{rem:cross-kernel-dependence}
The conclusion of Lemma~\ref{lem:cross-kernel} uses the subgroup structure of $G$ in an essential way: in the Klein-type case the two cyclic kernels $\langle\sigma\rangle$ and $\langle\sigma\rho^2\rangle$ have the \emph{same} image $\bar\sigma$ in $G/K_C \cong V_4$ (as $\sigma \cdot \sigma\rho^2 = \rho^2 \in K_C$), so both induce the same involution on $C$ and a single $2{+}2$ partition arises. In a group in which distinct order-$2$ subgroups had distinct images, the conclusion would have to allow several partitions.
\end{remark}

\subsection{Uniquely-kerneled $4$-orbit and configuration lemmas}

\begin{definition}\label{def:uniquely-kerneled}
A $4$-orbit $C^*$ in $P$ is \emph{uniquely-kerneled} if $K_{C'} \ne K_{C^*}$ for every other $4$-orbit $C'$ in $P$.
\end{definition}

\begin{lemma}\label{lem:exist-uniquely-kerneled}
Let $G$ act faithfully on $P$ with $|P| \le 13$ in a Type~C configuration. If no uniquely-kerneled $4$-orbit exists, then $m = 2$ and $K_C = K_D$.
\end{lemma}

\begin{proof}
We prove the contrapositive: if either $m \ne 2$, or $m = 2$ with the two $4$-orbits having distinct kernels, then a uniquely-kerneled $4$-orbit exists.

\emph{Case $m = 2$.} If $K_C \ne K_D$, both orbits are uniquely-kerneled.

\emph{Case $m = 3$.} The multiset $\{K_{C_1}, K_{C_2}, K_{C_3}\}$ falls into exactly one of three patterns.

\emph{Pattern (i): all three kernels equal.} Let $K = K_{C_1} = K_{C_2} = K_{C_3}$, an order-$2$ subgroup. The three $4$-orbits contribute $12$ elements, so $|P| - 12 \le 1$ and at most one further element remains. No $8$-orbit occurs (Type~C), no fourth $4$-orbit occurs (it would force $|P| \ge 16$), and no $2$-orbit occurs (it would require two further elements); so the only possible additional orbit is a single fixed point, with stabilizer $G \supseteq K$. Combined with $K_{C_i} = K$, we get $K \subseteq K_F$ for every orbit $F$, hence $K \subseteq \bigcap_F K_F$. Faithfulness gives $\bigcap_F K_F = \{e\}$, forcing $K = \{e\}$, contradicting $|K| = 2$. So Pattern~(i) does not occur.

\emph{Pattern (ii): exactly two kernels equal, one distinct.} The orbit with the distinct kernel is uniquely-kerneled.

\emph{Pattern (iii): all three kernels distinct.} Each $C_i$ is uniquely-kerneled.

In patterns (ii) and (iii) a uniquely-kerneled $4$-orbit exists.

\emph{Case $m \ge 4$.} This forces $|P| \ge 16$, contradicting $|P| \le 13$.
\end{proof}

The next lemma controls the configurations in which the only uniquely-kerneled $4$-orbit is of Klein type. Note that in a Type~C configuration $m \ge 2$, so a $4$-orbit other than $C^*$ always exists; the conclusion $m = 3$ below is therefore not vacuous.

\begin{lemma}\label{lem:klein-unique-config}
Let $G$ act faithfully on $P$ with $|P| \le 13$ in a Type~C configuration (so $m \ge 2$). Suppose a uniquely-kerneled Klein-type $4$-orbit $C^*$ exists, but no cyclic-type uniquely-kerneled $4$-orbit exists. Then:
\begin{enumerate}
\item[(i)] $m = 3$, with one Klein $4$-orbit (namely $C^*$) and two cyclic $4$-orbits sharing a common cyclic kernel;
\item[(ii)] $|P| \in \{12, 13\}$, with configuration $4+4+4$ (no extra orbits) or $4+4+4+1$ (one fixed point); in particular, $P$ has no $2$-orbit.
\end{enumerate}
\end{lemma}

\begin{proof}
Let $k$ denote the number of cyclic-type $4$-orbits in $P$. Since $C^*$ is uniquely-kerneled, no other $4$-orbit has kernel $\langle \rho^2 \rangle = K_{C^*}$; hence every $4$-orbit other than $C^*$ is cyclic, and $m = 1 + k$.

\emph{Step 1: $k \ge 2$.} Suppose $k \le 1$. If $k = 0$, then $m = 1$, contradicting $m \ge 2$ (Type~C). If $k = 1$, the unique cyclic $4$-orbit $C'$ has $K_{C'} \in \{\langle \sigma \rangle, \langle \sigma\rho^2 \rangle\}$, distinct from $K_{C^*} = \langle \rho^2 \rangle$; and $C^*$ is the only other $4$-orbit, so $K_{C'}$ differs from every other $4$-orbit's kernel. Hence $C'$ is a cyclic-type uniquely-kerneled $4$-orbit, contradicting the hypothesis.

\emph{Step 2: $k = 2$.} Since $4m \le |P| \le 13$, $m \le 3$. As $m = 1 + k$, $k \le 2$; with Step~1, $k = 2$ and $m = 3$.

\emph{Step 3: The two cyclic $4$-orbits share a common kernel.} Both cyclic kernels lie in $\{\langle \sigma \rangle, \langle \sigma\rho^2 \rangle\}$. Suppose for contradiction the two cyclic orbits $C_1, C_2$ have distinct kernels, so $\{K_{C_1}, K_{C_2}\} = \{\langle \sigma \rangle, \langle \sigma\rho^2 \rangle\}$. The only other $4$-orbit is $C^*$ with $K_{C^*} = \langle \rho^2 \rangle$. Then $K_{C_1}$ differs from $K_{C_2}$ (hypothesis) and from $K_{C^*}$ (the three subgroups $\langle \sigma \rangle, \langle \sigma\rho^2 \rangle, \langle \rho^2 \rangle$ are pairwise distinct), so $C_1$ is uniquely-kerneled and cyclic, contradicting the hypothesis. Hence $K_{C_1} = K_{C_2}$.

\emph{Step 4: Configurations.} The three $4$-orbits contribute $12$ elements, so the remaining elements (a union of orbits of size $1$ or $2$) number $|P| - 12 \le 1$; in particular no orbit of size $2$ fits. Thus $|P| = 12$ (configuration $4+4+4$) or $|P| = 13$ (configuration $4+4+4+1$, one fixed point), and $P$ has no $2$-orbit.
\end{proof}

\subsection{The case of a uniquely-kerneled $4$-orbit}

\begin{theorem}\label{thm:strategy1}
Let $G$ act faithfully on $P$ with $|P| \le 13$ in a Type~C configuration possessing a uniquely-kerneled $4$-orbit. Then $\Aut(P) \supsetneq G$.
\end{theorem}

\begin{proof}
Choose a uniquely-kerneled $4$-orbit $C^*$ as follows: if a cyclic-type uniquely-kerneled $4$-orbit exists, let $C^*$ be one such; otherwise let $C^*$ be a Klein-type uniquely-kerneled $4$-orbit (which then exists by hypothesis). Define $\tau \in \Sym(P)$ by $\tau|_{C^*} = \alpha$ and $\tau|_{P \setminus C^*} = \mathrm{id}$, with $\alpha \in \Sym(C^*)$ specified as follows.
\begin{enumerate}
\item[(i)] If $C^*$ is cyclic type, set $\alpha = (c_1\, c_3)$.
\item[(ii)] If $C^*$ is Klein type, then by the choice of $C^*$ no cyclic-type uniquely-kerneled $4$-orbit exists; the hypotheses of Lemma~\ref{lem:klein-unique-config} are therefore met, so the configuration is $4+4+4$ or $4+4+4+1$, the two other $4$-orbits are cyclic with a common kernel, and $P$ has no $2$-orbit. Writing the canonical partition as $\{P_1^{C^*}, P_2^{C^*}\} = \{\{u_0, u_1\}, \{v_0, v_1\}\}$, set $\alpha = (u_0\, u_1)$.
\end{enumerate}
In both cases $\alpha$ is an involution preserving $\{P_1^{C^*}, P_2^{C^*}\}$ setwise, fixing one class pointwise and exchanging the elements of the other.

We verify $\tau \in \Aut(P)$. Since $\tau = \mathrm{id}$ off $C^*$ and $C^*$ is an antichain, it suffices to show $\alpha$-invariance of $U_v^{C^*}$ and $D_v^{C^*}$ for every $v \in P \setminus C^*$. Consider the orbit of $v$.

If $v$ lies in another $4$-orbit $C'$, then $K_{C'} \ne K_{C^*}$ by unique-kerneledness, so Lemma~\ref{lem:cross-kernel} gives
\[
U_v^{C^*},\ D_v^{C^*} \in \{\emptyset, P_1^{C^*}, P_2^{C^*}, C^*\},
\]
all four $\alpha$-invariant.

If $v$ lies in a $2$-orbit, then $C^*$ is necessarily of cyclic type (by Lemma~\ref{lem:klein-unique-config}(ii), the Klein case admits no $2$-orbit), and $\Stab_G(v) \in \{\langle \rho \rangle, \langle \sigma\rho \rangle, \langle \sigma, \rho^2 \rangle\}$. Lemma~\ref{lem:stab-enum} gives $U_v^{C^*}, D_v^{C^*} \in \{\emptyset, P_1^{C^*}, P_2^{C^*}, C^*\}$, and all four candidate sets are preserved by $\alpha$, which fixes $\{P_1^{C^*}, P_2^{C^*}\}$ setwise.

If $v$ is a $G$-fixed point, $\Stab_G(v) = G$ acts transitively on $C^*$, so $U_v^{C^*}, D_v^{C^*} \in \{\emptyset, C^*\}$, both $\alpha$-invariant.

Thus $\tau \in \Aut(P)$. To see $\tau \notin G$: the image of $G$ in $\Sym(C^*)$ is $\langle \rho|_{C^*} \rangle$ in the cyclic case (sole non-identity involution $\rho^2|_{C^*} = (c_0\, c_2)(c_1\, c_3)$) and $V_4$ acting regularly in the Klein case (its three non-identity elements acting as double transpositions). The single transposition $\alpha$ lies in neither image. Hence $\Aut(P) \supsetneq G$.
\end{proof}

\subsection{The case $m = 2$ with $K_C = K_D$}

\begin{lemma}\label{lem:faithful-exclusion}
Suppose $G$ acts faithfully on a finite poset $P$ with $|P| \le 13$ in a Type~C configuration with $m = 2$ and $K_C = K_D$. Then:
\begin{enumerate}
\item[(i)] $K_C = K_D$ is of cyclic type;
\item[(ii)] at most one $2$-orbit $E$ in $P$ has $K_E = \langle \sigma, \rho^2 \rangle$.
\end{enumerate}
\end{lemma}

\begin{proof}
(i) If $K_C = K_D = \langle \rho^2 \rangle$, then $\rho^2$ acts trivially on $C$ and $D$. For any non-$4$-orbit $F$: if $F$ is a singleton, $K_F = G \ni \rho^2$; if $F$ is a $2$-orbit, $K_F$ is an order-$4$ subgroup, which contains $\rho^2$ by property~\hyperref[K2]{(K2)}. Hence $\rho^2 \in \bigcap_F K_F$, contradicting faithfulness (which requires $\bigcap_F K_F = \{e\}$). So $K_C = K_D$ is cyclic type.

(ii) Suppose $K_C = K_D = \langle \sigma \rangle$; the case $K_C = K_D = \langle \sigma\rho^2 \rangle$ follows by applying the automorphism of $G$ that fixes $\rho$ and sends $\sigma \mapsto \sigma\rho^2$, which swaps $\langle\sigma\rangle \leftrightarrow \langle\sigma\rho^2\rangle$ and fixes each of $\langle\rho^2\rangle$, $\langle\sigma,\rho^2\rangle$, $\langle\rho\rangle$, $\langle\sigma\rho\rangle$ setwise. For contradiction, assume two distinct $2$-orbits $E_1, E_2$ both have $K_{E_i} = \langle \sigma, \rho^2 \rangle$. Then $\sigma \in K_F$ for each $F \in \{C, D, E_1, E_2\}$, and these four orbits account for $4 + 4 + 2 + 2 = 12$ elements. Faithfulness forces $\bigcap_F K_F = \{e\}$ (intersection over all orbits), so some orbit $F^*$ has $\sigma \notin K_{F^*}$. We rule out every type for $F^*$ except a $2$-orbit: an $8$-orbit is excluded (Type~C); the only $4$-orbits are $C, D$, both with kernel $\langle \sigma \rangle \ni \sigma$; a singleton has kernel $G \ni \sigma$. So $F^*$ is a $2$-orbit with $K_{F^*} \in \{\langle \rho \rangle, \langle \sigma\rho \rangle\}$. Adjoining its two elements to the previous $12$ gives $|P| \ge 14$, contradicting $|P| \le 13$.
\end{proof}

\begin{lemma}\label{lem:pair-symmetry}
Let $C, D$ be cyclic-type $4$-orbits with $K_C = K_D$. Choose $\rho$-equivariant labelings $C = \{c_0, c_1, c_2, c_3\}$, $D = \{d_0, d_1, d_2, d_3\}$ (so $\rho \cdot c_k = c_{k+1}$ and $\rho \cdot d_k = d_{k+1}$, indices in $\mathbb{Z}_4$), and define the difference set
\[
S_{CD} := \{\,j - i \in \mathbb{Z}_4 : c_i < d_j\,\}.
\]
By $\rho$-equivariance, $c_i < d_j$ if and only if $j - i \in S_{CD}$. Define
\[
T(S_{CD}) := \{\,t \in \mathbb{Z}_4 : t - S_{CD} = S_{CD}\,\},
\]
the set of $t$ for which the reflection $x \mapsto t - x$ stabilizes $S_{CD}$. Then:
\begin{enumerate}
\item[(i)] $T(S_{CD}) = \mathbb{Z}_4$ if $|S_{CD}| \in \{0, 4\}$;
\item[(ii)] $T(S_{CD}) = \{0, 2\}$ if $S_{CD}$ is an opposite pair $\{a, a+2\}$;
\item[(iii)] $T(S_{CD}) = \{2a\}$ if $S_{CD} = \{a\}$ or $S_{CD} = \mathbb{Z}_4 \setminus \{a\}$ (so $T(S_{CD}) \subseteq \{0, 2\}$);
\item[(iv)] $T(S_{CD}) = \{a + b\} \subseteq \{1, 3\}$ if $S_{CD} = \{a, b\}$ is adjacent (i.e., $b - a \in \{1, 3\}$).
\end{enumerate}
In particular, $T(S_{CD})$ contains an even element unless $S_{CD}$ is adjacent.
\end{lemma}

\begin{proof}
Each part is a direct check of $t - S_{CD} = S_{CD}$ over $t \in \mathbb{Z}_4$. The parts are exhaustive: subsets of size $0$ or $4$ fall under (i); singletons and their complements under (iii); and the three two-element subsets are the opposite pairs $\{0,2\}$, $\{1,3\}$ (case (ii)) and the four adjacent pairs $\{0,1\},\{1,2\},\{2,3\},\{0,3\}$ (case (iv)). For the last assertion, the even elements of $\mathbb{Z}_4$ are $\{0,2\}$: in (i)--(iii) $T(S_{CD})$ meets $\{0,2\}$, while in (iv) $a+b$ is odd whenever $b - a \in \{1,3\}$.
\end{proof}

\begin{remark}\label{rem:labeling}
The difference set $S_{CD}$ depends on the chosen $\rho$-equivariant labelings only up to translation: re-indexing the origin of $C$ or of $D$ replaces $S_{CD}$ by $S' = S_{CD} + c$ for some $c \in \mathbb{Z}_4$. This replaces $T(S_{CD})$ by an \emph{even} translate: $t' \in T(S')$ unwinds, upon subtracting $c$ from both sides of $t' - (S_{CD}+c) = S_{CD}+c$, to $(t' - 2c) - S_{CD} = S_{CD}$, i.e., $T(S') = T(S_{CD}) + 2c$ with $2c \in \{0, 2\}$. Consequently the dichotomy used in Theorem~\ref{thm:strategy2}, namely whether $T(S_{CD})$ meets $\{0,2\}$ (Case~A there) or is contained in $\{1,3\}$ (Case~B), is independent of the labeling. This is the only labeling-independence the construction requires: the construction itself is carried out for one fixed labeling. We also note that $T(S_{CD})$ need not be a subgroup of $\mathbb{Z}_4$ (reflections compose to translations); we treat it purely as the set of valid reflection parameters.
\end{remark}

\begin{lemma}\label{lem:adjacency-obstruction}
Under the assumptions of Lemma~\ref{lem:pair-symmetry} with $S_{CD}$ adjacent, suppose $z \in P \setminus (C \cup D)$ has $\Stab_G(z) = \langle \sigma, \rho^2 \rangle$. Then:
\begin{enumerate}
\item[(i)] $U_z^C \in \{P_1^C, P_2^C\}$ implies $U_z^D = D$;
\item[(ii)] $D_z^D \in \{P_1^D, P_2^D\}$ implies $D_z^C = C$.
\end{enumerate}
\end{lemma}

\begin{proof}
By Lemma~\ref{lem:equivariance}, $U_z^C, D_z^C \subseteq C$ and $U_z^D, D_z^D \subseteq D$ are $\langle \sigma, \rho^2 \rangle$-invariant; by Lemma~\ref{lem:stab-enum} each lies in $\{\emptyset, P_1^B, P_2^B, B\}$ for the corresponding orbit $B \in \{C, D\}$. Write $S_{CD} = \{a, a+1\}$; then each $c_i$ is below exactly $\{d_{i+a}, d_{i+a+1}\}$, and each $d_j$ has exactly $\{c_{j-a}, c_{j-a-1}\}$ below it.

(i) Suppose $U_z^C = P_1^C = \{c_0, c_2\}$ (the case $P_2^C$ is analogous). By transitivity,
\[
z < c_0 \implies z < d_a, d_{a+1}; \qquad z < c_2 \implies z < d_{a+2}, d_{a+3}.
\]
So $U_z^D \supseteq \{d_a, d_{a+1}, d_{a+2}, d_{a+3}\} = D$, hence $U_z^D = D$.

(ii) Suppose $D_z^D = P_1^D = \{d_0, d_2\}$ (the case $P_2^D$ is analogous). By transitivity, every $c_i$ below some $d_j \in D_z^D$ satisfies $c_i < z$. Now
\[
\{c_i : c_i < d_0\} = \{c_{-a-1}, c_{-a}\}, \qquad \{c_i : c_i < d_2\} = \{c_{1-a}, c_{2-a}\},
\]
whose union has index set $\{-a-1, -a, -a+1, -a+2\}$, a translate of $\{0,1,2,3\}$ and hence all of $\mathbb{Z}_4$. So $D_z^C \supseteq C$, giving $D_z^C = C$.
\end{proof}

\begin{lemma}\label{lem:mutex}
Under the assumptions of Lemma~\ref{lem:pair-symmetry} with $S_{CD}$ adjacent, suppose $z \in P \setminus (C \cup D)$ has $\Stab_G(z) = \langle \sigma, \rho^2 \rangle$. Say $C$ is \emph{constrained at $z$} if $U_z^C \in \{P_1^C, P_2^C\}$ or $D_z^C \in \{P_1^C, P_2^C\}$, and analogously for $D$. Then $C$ and $D$ cannot both be constrained at $z$.
\end{lemma}

\begin{proof}
Suppose both are constrained. Then one of $U_z^C, D_z^C$ lies in $\{P_1^C, P_2^C\}$ and one of $U_z^D, D_z^D$ lies in $\{P_1^D, P_2^D\}$, giving four cases.

\emph{(UU).} $U_z^C \in \{P_1^C, P_2^C\}$ forces $U_z^D = D$ (Lemma~\ref{lem:adjacency-obstruction}(i)), contradicting $U_z^D \in \{P_1^D, P_2^D\}$.

\emph{(DD).} $D_z^D \in \{P_1^D, P_2^D\}$ forces $D_z^C = C$ (Lemma~\ref{lem:adjacency-obstruction}(ii)), contradicting $D_z^C \in \{P_1^C, P_2^C\}$.

\emph{(UD).} $U_z^C \in \{P_1^C, P_2^C\}$ gives $U_z^D = D$; $D_z^D \in \{P_1^D, P_2^D\}$ gives $D_z^C = C$. Pick $c \in U_z^C$ (nonempty, $|U_z^C| = 2$): then $z < c$, while $D_z^C = C$ gives $c < z$, contradicting irreflexivity.

\emph{(DU).} $D_z^C \in \{P_1^C, P_2^C\}$ and $U_z^D \in \{P_1^D, P_2^D\}$, each of size $2$. For every $(c, d) \in D_z^C \times U_z^D$, transitivity gives $c < z < d$, hence $c < d$, contributing $j - i$ to $S_{CD}$. The four sub-cases:
$D_z^C = P_1^C, U_z^D = P_1^D$ gives differences $\{0,2,2,0\} \pmod{4} = \{0,2\}$;
$P_1^C, P_2^D$ gives $\{1,3,3,1\} \pmod{4} = \{1,3\}$;
$P_2^C, P_1^D$ gives $\{3,1,1,3\} \pmod{4} = \{1,3\}$;
$P_2^C, P_2^D$ gives $\{0,2\}$.
In each, the difference set $\{0,2\}$ or $\{1,3\}$ is contained in $S_{CD}$; as $|S_{CD}| = 2$, this forces $S_{CD} \in \{\{0,2\}, \{1,3\}\}$, neither adjacent, contradicting adjacency of $S_{CD}$.

All four cases are contradictory.
\end{proof}

\begin{remark}\label{rem:mirror}
Lemmas~\ref{lem:adjacency-obstruction} and~\ref{lem:mutex} are stated for the difference set $S_{CD}$ and cross-relations of the form $c < d$. The mirror statements, obtained for $S_{DC} := \{\,i - j \in \mathbb{Z}_4 : d_j < c_i\,\}$ and cross-relations $d < c$, hold as well: interchanging the roles of $C$ and $D$ throughout transposes the statements and proofs verbatim. We appeal to these mirror versions when relabeling $C$ and $D$ in Theorem~\ref{thm:strategy2}.
\end{remark}

\begin{lemma}\label{lem:cd-disjoint}
Let $C, D$ be distinct cyclic-type $4$-orbits with $\rho$-equivariant labelings as in Lemma~\ref{lem:pair-symmetry}, and set $S_{DC} := \{\,i - j \in \mathbb{Z}_4 : d_j < c_i\,\}$. Then at most one of $S_{CD}$, $S_{DC}$ is non-empty.
\end{lemma}

\begin{proof}
If both are non-empty, $\rho$-equivariance gives $s, r$ with $c_0 < d_s$ and $d_0 < c_r$; applying $\rho^s$ to the second gives $d_s < c_{r+s}$, and transitivity yields $c_0 < c_{r+s}$. If $r + s \ne 0$ this is a strict relation between two distinct elements of $C$, contradicting the antichain property (Lemma~\ref{lem:antichain}); if $r + s = 0$ it reads $c_0 < c_0$, contradicting irreflexivity. Either way, a contradiction.
\end{proof}

\begin{theorem}\label{thm:strategy2}
Let $G$ act faithfully on a finite poset $P$ with $|P| \le 13$, with no orbit of size $8$, exactly two orbits $C, D$ of size $4$, and $K_C = K_D$. Then $\Aut(P) \supsetneq G$.
\end{theorem}

\begin{proof}
By Lemma~\ref{lem:faithful-exclusion}(i), $K_C = K_D \in \{\langle \sigma \rangle, \langle \sigma\rho^2 \rangle\}$ is cyclic type; in particular both $C$ and $D$ are cyclic-type orbits, and we fix $\rho$-equivariant labelings as in Lemma~\ref{lem:pair-symmetry}. By Lemma~\ref{lem:cd-disjoint}, at most one of $S_{CD}, S_{DC}$ is non-empty. The construction below and all its conditions are symmetric under interchanging $C \leftrightarrow D$ (swapping $t_C \leftrightarrow t_D$ and $S_{CD} \leftrightarrow S_{DC}$, and invoking the mirror versions of Lemmas~\ref{lem:adjacency-obstruction} and~\ref{lem:mutex} recorded in Remark~\ref{rem:mirror}), so after relabeling if necessary we may assume $S_{DC} = \emptyset$. (If both are empty the relabeling is vacuous and the argument below applies unchanged with $S_{CD} = \emptyset$, which falls under case (i) of Lemma~\ref{lem:pair-symmetry}.) Thus the only cross-orbit comparabilities are $c < d$, recorded by $S_{CD}$.

Note that since $m = 2$ and there is no $8$-orbit, every $v \in P \setminus (C \cup D)$ lies in an orbit of size $1$ or $2$; its stabilizer is therefore $G$ or one of the three order-$4$ subgroups $\langle\rho\rangle$, $\langle\sigma\rho\rangle$, $\langle\sigma,\rho^2\rangle$, so Lemma~\ref{lem:stab-enum} applies to $v$ with respect to each of the cyclic-type orbits $C$ and $D$.

Construct $\tau \in \Sym(P)$ by $\tau|_C = \alpha_C$, $\tau|_D = \alpha_D$, $\tau|_{P \setminus (C \cup D)} = \mathrm{id}$, where $\alpha_C(x) = t_C - x$ and $\alpha_D(y) = t_D - y$ for parameters $t_C, t_D \in \mathbb{Z}_4$ chosen below. Each of $\alpha_C, \alpha_D$ is an involution, so $\tau$ is a bijection and it suffices to check that $\tau$ preserves comparability, i.e., the two conditions:
\begin{itemize}
\item[(C1)] $(\alpha_C \times \alpha_D)$ preserves $\{(c, d) : c < d\}$, equivalently $t_D - t_C \in T(S_{CD})$ (Lemma~\ref{lem:pair-symmetry}); as $S_{DC} = \emptyset$, there are no $d < c$ relations, so this is the complete cross-orbit condition;
\item[(C2)] for each $v \in P \setminus (C \cup D)$, $\alpha_C$ preserves $U_v^C, D_v^C$ and $\alpha_D$ preserves $U_v^D, D_v^D$.
\end{itemize}

To see the equivalence asserted in (C1): the pair $(c_i, d_j)$ is sent by $\alpha_C \times \alpha_D$ to $(c_{t_C - i}, d_{t_D - j})$, whose index difference is $(t_D - j) - (t_C - i) = (t_D - t_C) - (j - i)$. Thus $\alpha_C \times \alpha_D$ maps the relation with difference $\delta$ to the relation with difference $(t_D - t_C) - \delta$, and it preserves $\{(c,d) : c < d\}$ precisely when the reflection $\delta \mapsto (t_D - t_C) - \delta$ stabilizes $S_{CD}$, i.e., when $t_D - t_C \in T(S_{CD})$ (Lemma~\ref{lem:pair-symmetry}).

A direct computation shows $x \mapsto t - x$ fixes each of the classes $P_1^C, P_2^C$ iff $t$ is even (and swaps them iff $t$ is odd). Since $x \mapsto t-x$ is a bijection of $C$, it always preserves $\emptyset$ and $C$; thus (C2) on $C$ is automatic unless some $U_v^C$ or $D_v^C$ actually equals $P_1^C$ or $P_2^C$, in which case it amounts to requiring $t_C$ even. Accordingly, call $\alpha_C$ \emph{constrained} if some $v \in P \setminus (C \cup D)$ has $U_v^C \in \{P_1^C, P_2^C\}$ or $D_v^C \in \{P_1^C, P_2^C\}$, and define $\alpha_D$ \emph{constrained} analogously. With this terminology, (C2) holds if and only if $t_C$ is even whenever $\alpha_C$ is constrained, and $t_D$ is even whenever $\alpha_D$ is constrained.

By Lemma~\ref{lem:stab-enum}, $U_v^C, D_v^C \in \{\emptyset, P_1^C, P_2^C, C\}$ for every $v \in P \setminus (C \cup D)$, with the partition values possible only when $\Stab_G(v) = \langle \sigma, \rho^2 \rangle$; the same holds with $D$ in place of $C$. By Lemma~\ref{lem:faithful-exclusion}(ii), at most one $2$-orbit $E_*$ has stabilizer $\langle \sigma, \rho^2 \rangle$. Hence every witness $v$ to the constraint of $\alpha_C$ or of $\alpha_D$ lies in $E_*$; in particular, if $E_*$ does not exist, then neither $\alpha_C$ nor $\alpha_D$ is constrained.

Suppose now that $E_*$ exists and fix $z_* \in E_*$. We claim that $\alpha_C$ is constrained if and only if $C$ is constrained \emph{at $z_*$} in the sense of Lemma~\ref{lem:mutex}, and similarly for $\alpha_D$ and $D$; that is, the global notion can be tested at the single point $z_*$. Indeed, the only other element of $E_*$ is $g \cdot z_*$ for any $g \in G \setminus \Stab_G(z_*)$, and Lemma~\ref{lem:equivariance} gives $U_{g z_*}^{C} = g \cdot U_{z_*}^{C}$ and $D_{g z_*}^{C} = g \cdot D_{z_*}^{C}$. Every $g \in G$ acts on $C$ through the translation group $\langle \rho|_C \rangle$, and a translation of $\mathbb{Z}_4$ carries the canonical partition $\{P_1^C, P_2^C\}$ to itself setwise (fixing or swapping the two classes according to the parity of the translation). Hence $U_{g z_*}^{C} \in \{P_1^C, P_2^C\}$ if and only if $U_{z_*}^{C} \in \{P_1^C, P_2^C\}$, and likewise for $D_{\cdot}^{C}$; so a witness exists at one point of $E_*$ if and only if one exists at the other, proving the claim. The same argument applies verbatim to $D$.

By Lemma~\ref{lem:pair-symmetry}, $T(S_{CD})$ contains an even element unless $S_{CD}$ is adjacent.

\emph{Case A: $T(S_{CD}) \cap \{0,2\} \ne \emptyset$.} Take $t_C = 0$ and $t_D \in T(S_{CD}) \cap \{0, 2\}$; both even, so (C2) holds automatically and (C1) holds by construction.

\emph{Case B: $T(S_{CD}) \subseteq \{1, 3\}$.} Then $S_{CD}$ is adjacent and $T(S_{CD}) = \{2q + 1\}$ for some $q \in \mathbb{Z}_4$, so $t_D - t_C$ must be the odd value $2q+1$; exactly one of $t_C, t_D$ is even.

\emph{(I) $E_*$ does not exist.} Neither $\alpha_C$ nor $\alpha_D$ is constrained, as noted above. Take $(t_C, t_D) = (0, 2q + 1)$: (C1) and (C2) hold.

\emph{(II) $E_*$ exists.} By the single-point criterion above and Lemma~\ref{lem:mutex} (whose adjacency hypothesis holds in Case~B), $C$ and $D$ are not both constrained at $z_*$, so at most one of $\alpha_C, \alpha_D$ is constrained.

\emph{(II.a) Neither constrained.} Take $(t_C, t_D) = (0, 2q + 1)$ as in (I).

\emph{(II.b) $\alpha_C$ constrained, $\alpha_D$ not.} Take $(t_C, t_D) = (0, 2q + 1)$: $t_C = 0$ is even (satisfying $\alpha_C$'s constraint), $\alpha_D$ is unconstrained, and $t_D - t_C = 2q + 1 \in T(S_{CD})$.

\emph{(II.c) $\alpha_D$ constrained, $\alpha_C$ not.} Take $(t_C, t_D) = (3 - 2q, 0)$. Then $t_D = 0$ is even (satisfying $\alpha_D$'s constraint). For the parity of $t_C$: since $2q$ is even and $3$ is odd, $3 - 2q$ is always odd, so $t_C$ is always odd and $\alpha_C$ (unconstrained) has no parity requirement to violate. For (C1):
\[
t_D - t_C = -(3 - 2q) = 2q - 3 \equiv 2q + 1 \pmod 4,
\]
so $t_D - t_C \in T(S_{CD})$. Hence (C1) and (C2) hold.

It remains to verify $\tau \notin G$. The image of $G$ in $\Sym(C)$ consists entirely of translations $c_k \mapsto c_{k+\ell}$, $\ell \in \mathbb{Z}_4$, since $G/K_C \cong \mathbb{Z}_4$ acts regularly on $C$. If $\alpha_C = \rho^\ell|_C$, then $t_C - x = x + \ell$ for all $x$, i.e., $2x = t_C - \ell$ for all $x \in \mathbb{Z}_4$, impossible (the left side takes only two values, $0$ and $2$, in $\mathbb{Z}_4$). So $\alpha_C$ is not in the image of $G$ on $C$; but $\tau \in G$ would force $\tau|_C$ into that image. Hence $\tau \notin G$ and $\Aut(P) \supsetneq G$.
\end{proof}

\subsection{Conclusion of the Type~C analysis}

\begin{theorem}\label{thm:typeC-main}
Let $G$ act faithfully on a finite poset $P$ with $|P| \le 13$, with at least two $4$-orbits and no $8$-orbit. Then $\Aut(P) \supsetneq G$.
\end{theorem}

\begin{proof}
Let $m \ge 2$ be the number of $4$-orbits. By Lemma~\ref{lem:exist-uniquely-kerneled}, either a uniquely-kerneled $4$-orbit exists, or $m = 2$ with $K_C = K_D$. The former is handled by Theorem~\ref{thm:strategy1}, the latter by Theorem~\ref{thm:strategy2}; in both an element of $\Aut(P) \setminus G$ is produced.
\end{proof}

\section{The lower bound}\label{sec:lower}

\begin{theorem}\label{thm:lower}
For all $n \le 13$, no $n$-element poset $P$ has $\Aut(P) \cong G$.
\end{theorem}

\begin{proof}
For $n \le 4$, any subgroup of $\Sym(P)$ embeds into $S_n$, and a copy of $G$ inside $\Aut(P) \le \Sym(P)$ would embed $G$ into $S_n$, impossible by Lemma~\ref{lem:S4-embed}. So assume $5 \le n \le 13$ and, for contradiction, $\Aut(P) \cong G$. Fixing an isomorphism realizes $G$ as a subgroup of $\Aut(P) \subseteq \Sym(P)$ of order $8$, acting faithfully on $P$ by order-automorphisms. By the orbit--stabilizer theorem every orbit has size in $\{1, 2, 4, 8\}$, so the action is of one of Types~A--D.

If Type~D, then $\rho^2$ acts trivially on $P$ (Lemma~\ref{lem:typeD}), contradicting faithfulness. If Type~A with $P$ a single $8$-orbit, Lemma~\ref{lem:pure-8orbit} gives $\Aut(P) \cong S_8$, and $|S_8| \ne 8$ contradicts $\Aut(P) \cong G$ directly. In each of the remaining cases, namely Type~A with an $8$-orbit and a further orbit (Lemma~\ref{lem:8orbit}), Type~B (Lemma~\ref{lem:hiddeninv}), and Type~C (Theorem~\ref{thm:typeC-main}), the cited result produces $\phi \in \Aut(P) \setminus G$, whence $|\Aut(P)| \ge 9 > 8 = |G|$, again contradicting $\Aut(P) \cong G$. In every case we reach a contradiction, so no such $P$ exists.
\end{proof}

\section{The construction and the main theorem}\label{sec:construction}

\begin{definition}\label{def:p14}
Let $P_{14} = \{a_0, a_1\} \cup \{(i, j) : i \in \mathbb{Z}_4,\ j \in \{0, 1, 2\}\}$, equipped with the partial order generated by the following relations (that is, the reflexive-transitive closure of (R1)--(R3)):
\begin{itemize}
\item[(R1)] $(i, 0) < (i, 1)$ and $(i, 1) < (i, 2)$ for $i \in \mathbb{Z}_4$;
\item[(R2)] $(i, 0) < (i+1, 2)$ for $i \in \mathbb{Z}_4$ (indices mod $4$);
\item[(R3)] $a_k < (i, 0)$ for $k \in \{0, 1\}$, $i \in \mathbb{Z}_4$.
\end{itemize}
Lemma~\ref{lem:p14-covers} below verifies that (R1)--(R3) are precisely the covering relations of $P_{14}$, so that Figure~\ref{fig:hasse} is its Hasse diagram.
\end{definition}

\begin{remark}\label{rem:ordinalsum}
$P_{14}$ is precisely the ordinal sum of the $2$-element antichain $\{a_0, a_1\}$ with the $12$-point poset $\mathbb{Z}_4 \times \{0,1,2\}$, $(i,2) > (i,1) > (i,0) < (i+1,2)$, realizing $\mathbb{Z}_4$ (cf.\ \cite[proof of Prop.~2.3]{BarmakBarreto2024} and \cite{Frucht1950}). We nevertheless determine $\Aut(P_{14})$ directly, to keep the paper self-contained.
\end{remark}

\begin{lemma}\label{lem:p14-covers}
The strict order relations of $P_{14}$ are exactly the following:
\begin{enumerate}
\item[(i)] $a_k < (i, j)$ for all $k \in \{0,1\}$, $i \in \mathbb{Z}_4$, $j \in \{0,1,2\}$;
\item[(ii)] $(i,0) < (i,1)$, $(i,1) < (i,2)$, $(i,0) < (i,2)$, and $(i,0) < (i+1,2)$ for all $i \in \mathbb{Z}_4$.
\end{enumerate}
Writing $U_x$ (resp.\ $D_x$) for the set of elements strictly above (resp.\ below) $x$, it follows that for each $i \in \mathbb{Z}_4$,
\[
U_{(i,0)} = \{(i,1), (i,2), (i+1,2)\}, \qquad U_{(i,1)} = \{(i,2)\}, \qquad U_{(i,2)} = \emptyset,
\]
\[
\begin{gathered}
D_{(i,0)} = \{a_0, a_1\}, \qquad D_{(i,1)} = \{(i,0), a_0, a_1\}, \\
D_{(i,2)} = \{(i,1), (i,0), (i-1,0), a_0, a_1\}.
\end{gathered}
\]
Moreover, the covering relations of $P_{14}$ are precisely (R1)--(R3): no relation in the list (R1)--(R3) is implied by transitivity from the others, and no unlisted cover arises.
\end{lemma}

\begin{proof}
Let $R$ denote the set of strict relations listed in (i)--(ii). Every relation in $R$ follows from (R1)--(R3) by transitivity: the relations in $R$ not already in (R1)--(R3) are $(i,0) < (i,2)$ (via $(i,0) < (i,1) < (i,2)$) and $a_k < (i,1)$, $a_k < (i,2)$ (via $a_k < (i,0)$ and (ii)). Conversely, $R$ is irreflexive and transitive, hence contains the transitive closure of (R1)--(R3): a composable pair of relations in $R$ either starts at some $a_k$, in which case the composite is again of type (i), or starts at some $(i,0)$, in which case it ends at $(i,2)$ or $(i+1,2)$, both already in (ii). Indeed, relations ending at level $j = 2$ cannot be extended, and the relation $(i,0) < (i,1)$ extends only to $(i,0) < (i,2)$. Hence $R$ is exactly the strict order of $P_{14}$, and the displayed up-sets and down-sets follow by inspection of the list.

For the covering relations, each relation in (R1)--(R3) has empty open interval, by the up- and down-sets above:
\[
U_{(i,0)} \cap D_{(i,1)} = \emptyset, \qquad
U_{(i,1)} \cap D_{(i,2)} = \emptyset,
\]
\[
U_{(i,0)} \cap D_{(i+1,2)} = \{(i,1), (i,2), (i+1,2)\} \cap \{(i+1,1), (i+1,0), (i,0), a_0, a_1\} = \emptyset,
\]
and $D_{(i,0)} = \{a_0, a_1\}$ with $a_0, a_1$ incomparable, so $a_k \prec (i,0)$. Conversely, the strict relations of $R$ that are not listed in (R1)--(R3) are not covers: $(i,0) < (i,1) < (i,2)$, $a_k < (i,0) < (i,1)$, $a_k < (i,0) < (i,2)$, and $a_k < (i,0) < (i+1,2)$ exhibit intermediate elements. Hence the covers of $P_{14}$ are exactly (R1)--(R3).
\end{proof}

\begin{figure}[h]
\centering
\begin{tikzpicture}[scale=0.9, every node/.style={inner sep=1pt, font=\small}]
\foreach \i in {0,1,2,3} {
  \node (a\i 0) at (\i*1.8, 0) {$(\i,0)$};
  \node (a\i 1) at (\i*1.8, 1.2) {$(\i,1)$};
  \node (a\i 2) at (\i*1.8, 2.4) {$(\i,2)$};
  \draw (a\i 0) -- (a\i 1);
  \draw (a\i 1) -- (a\i 2);
}
\draw (a00) -- (a12);
\draw (a10) -- (a22);
\draw (a20) -- (a32);
\draw (a30) .. controls (7.2,0.9) and (7.2,3.6) .. (2.7,3.8)
            .. controls (-1.8,4.0) and (-1.8,2.9) .. (a02);
\node (a0) at (0.9, -1.5) {$a_0$};
\node (a1) at (4.5, -1.5) {$a_1$};
\foreach \i in {0,1,2,3} {
  \draw (a0) -- (a\i 0);
  \draw (a1) -- (a\i 0);
}
\end{tikzpicture}
\caption{The Hasse diagram of $P_{14}$. The poset is invariant under the shift $\rho \colon (i,j) \mapsto (i+1,j)$ (first coordinate mod $4$) and under the exchange $a_0 \leftrightarrow a_1$; the wrap-around covering relation $(3,0) \prec (0,2)$ is routed around the outside of the diagram.}
\label{fig:hasse}
\end{figure}
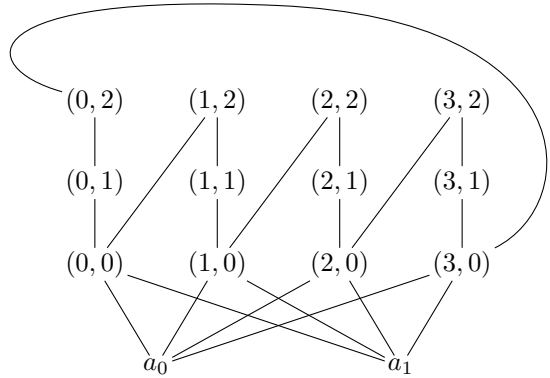

\begin{theorem}\label{thm:aut-p14}
$\Aut(P_{14}) \cong \mathbb{Z}_2 \times \mathbb{Z}_4$.
\end{theorem}

\begin{proof}
\emph{Step 1.} Define $\rho, \sigma \in \Sym(P_{14})$ by
\[
\rho(a_k) = a_k,\ \rho((i, j)) = (i+1, j); \qquad \sigma(a_k) = a_{k+1 \bmod 2},\ \sigma((i, j)) = (i, j).
\]
Each of $\rho$ and $\sigma$ maps the set of generating relations (R1)--(R3) bijectively onto itself ($\rho$ shifts the index $i$ in each of (R1)--(R3), and $\sigma$ permutes only the relations (R3), exchanging $a_0$ and $a_1$), and the same holds for their inverses $\rho^3$ and $\sigma$. Since the order of $P_{14}$ is the reflexive-transitive closure of (R1)--(R3), both maps and their inverses are order-preserving, i.e., $\rho, \sigma \in \Aut(P_{14})$. One checks $\rho^4 = \mathrm{id}$, $\sigma^2 = \mathrm{id}$, $\rho\sigma = \sigma\rho$; that $\rho$ has order exactly $4$ (it cycles the first coordinate); and that $\sigma \notin \langle \rho \rangle$ (since $\rho$ fixes $a_0, a_1$ while $\sigma$ exchanges them). Hence $\langle \rho \rangle \cap \langle \sigma \rangle = \{\mathrm{id}\}$ and $\langle \rho, \sigma \rangle$ has order $8$; being generated by two commuting elements of orders $4$ and $2$ with trivial intersection, $\langle \rho, \sigma \rangle \cong \mathbb{Z}_2 \times \mathbb{Z}_4$ is a subgroup of $\Aut(P_{14})$.

\emph{Step 2: Heights.} By Lemma~\ref{lem:p14-covers},
\[
h(a_k) = 0,\quad h((i, 0)) = 1,\quad h((i, 1)) = 2,\quad h((i, 2)) = 3,
\]
the last realized by the chain $a_k \prec (i, 0) \prec (i, 1) \prec (i, 2)$ (the alternative chain $a_k \prec (i-1, 0) \prec (i, 2)$ has length $2$). As noted in Section~\ref{sec:prelim}, every $\phi \in \Aut(P_{14})$ preserves $h$, hence preserves each level $h^{-1}(\ell)$ setwise.

\emph{Step 3.} Let $\phi \in \Aut(P_{14})$. By Step~2, $\phi$ permutes each level; define $\pi, \pi', \pi'' \in \Sym(\mathbb{Z}_4)$ and $\varepsilon \in \mathbb{Z}_2$ by
\[
\phi(a_k) = a_{k + \varepsilon},\ \phi((i, 0)) = (\pi(i), 0),\ \phi((i, 1)) = (\pi'(i), 1),\ \phi((i, 2)) = (\pi''(i), 2).
\]
We use throughout that an order-isomorphism of a finite poset maps covering relations onto covering relations, and that Lemma~\ref{lem:p14-covers} identifies the covers of $P_{14}$.

(i) $\pi' = \pi$: from $(i,0) \prec (i,1)$, applying $\phi$ gives $(\pi(i),0) \prec (\pi'(i),1)$. By Lemma~\ref{lem:p14-covers}, the covers of $(\pi(i),0)$ are $(\pi(i),1)$ and $(\pi(i)+1,2)$, of heights $2$ and $3$ respectively, so the unique height-$2$ cover of $(\pi(i),0)$ is $(\pi(i),1)$, giving $\pi'(i) = \pi(i)$.

(ii) $\pi'' = \pi$: from $(i,1) \prec (i,2)$, applying $\phi$ gives $(\pi(i),1) \prec (\pi''(i),2)$; by Lemma~\ref{lem:p14-covers}, the only cover of $(\pi(i),1)$ is $(\pi(i),2)$, so $\pi''(i) = \pi(i)$.

(iii) $\pi(i+1) = \pi(i) + 1$: the relation $(i,0) \prec (i+1,2)$ gives $(\pi(i),0) \prec (\pi(i+1),2)$ (using $\pi'' = \pi$). By Lemma~\ref{lem:p14-covers}, the covers of $(b,0)$ are $(b,1)$ and $(b+1,2)$, so the unique height-$3$ cover of $(b,0)$ is $(b+1,2)$. Hence $(\pi(i+1),2) = (\pi(i)+1,2)$, giving $\pi(i+1) = \pi(i)+1$.

(iv) Iterating (iii), $\pi(i) = \pi(0) + i$: $\pi$ is the translation by $\pi(0) \in \mathbb{Z}_4$.

\emph{Step 4: Counting.} By Step~3 each $\phi$ is determined by $(\pi(0), \varepsilon) \in \mathbb{Z}_4 \times \mathbb{Z}_2$: $\pi(0)$ governs all three levels $h^{-1}(1), h^{-1}(2), h^{-1}(3)$ (as $\pi'' = \pi' = \pi$ and $\pi(i) = \pi(0)+i$), and $\varepsilon$ governs $h^{-1}(0) = \{a_0, a_1\}$. The assignment $\phi \mapsto (\pi(0), \varepsilon)$ is injective, so $|\Aut(P_{14})| \le 8$. With Step~1, $\Aut(P_{14}) = \langle \rho, \sigma \rangle \cong \mathbb{Z}_2 \times \mathbb{Z}_4$.
\end{proof}

\begin{theorem}\label{thm:main}
The minimal size of a poset whose automorphism group is isomorphic to $\mathbb{Z}_2 \times \mathbb{Z}_4$ is $14$, i.e.,
\[
\beta(\mathbb{Z}_2 \times \mathbb{Z}_4) = 14.
\]
\end{theorem}

\begin{proof}
For the upper bound, Theorem~\ref{thm:aut-p14} provides a $14$-element poset with automorphism group isomorphic to $G$, so $\beta(G) \le 14$ (this also follows from Propositions~\ref{prop:product-bound} and~\ref{prop:cyclic}, in view of Remark~\ref{rem:ordinalsum}). For the lower bound, Theorem~\ref{thm:lower} shows no poset on at most $13$ elements has automorphism group isomorphic to $G$, so $\beta(G) \ge 14$. Combining, $\beta(G) = 14$.
\end{proof}

\begin{remark}\label{rem:nonunique}
We make no claim that $P_{14}$ is the unique $14$-element poset realizing $\mathbb{Z}_2 \times \mathbb{Z}_4$. Its order-dual (obtained by reversing all relations) is also a realizer, since the automorphism group is invariant under order-duality; and it is genuinely distinct, as $P_{14}$ has two minimal and four maximal elements whereas its dual has four minimal and two maximal. Only the cardinality $\beta(\mathbb{Z}_2 \times \mathbb{Z}_4) = 14$ is canonical.
\end{remark}

\section{Concluding remarks and open problems}\label{sec:remarks}

The methods developed here (the decomposition into orbit types A--D, the analysis of $4$-orbits via their kernels, and the parity-and-adjacency constraints for cyclic configurations) are closely adapted to the subgroup lattice of $\mathbb{Z}_2 \times \mathbb{Z}_4$: the proof exploits that $G$ has exactly three subgroups each of orders $2$ and $4$, with the intersection properties~\hyperref[K1]{(K1)} and~\hyperref[K2]{(K2)}. Extending the analysis to $\mathbb{Z}_2 \times \mathbb{Z}_{2^k}$ for $k \ge 3$ would require controlling interactions between orbits of several sizes $2^j$ simultaneously, and lies beyond the scope of this work.

Several natural problems remain open. First, for the two non-abelian groups of order $8$, the dihedral group $D_4$ and the quaternion group $Q_8$, the exact value of $\beta$ is unknown to us; the orbit-type strategy of Sections~\ref{sec:smalltypes}--\ref{sec:typeC} may serve as a template, but the non-abelian setting loses the constancy of stabilizers along orbits (Section~\ref{subsec:fourorbits}) on which much of the present analysis rests. Second, it would be interesting to determine for which groups the product bound of Proposition~\ref{prop:product-bound} is attained: it can fail already for cyclic factors, since $\beta(\mathbb{Z}_{12}) = \beta(\mathbb{Z}_3) + \beta(\mathbb{Z}_4) - 1$ by~\cite{BarmakBarreto2024}, whereas Theorem~\ref{thm:intro-main} and Proposition~\ref{prop:elementary} show that it is attained for $\mathbb{Z}_2 \times \mathbb{Z}_4$, $\mathbb{Z}_2 \times \mathbb{Z}_2$, and $\mathbb{Z}_2^3$. Finally, the determination of $\beta(G)$ for all finite abelian groups $G$, posed implicitly by~\cite{Barreto2021, BarmakBarreto2024}, remains open even for abelian $2$-groups.

\section*{Declarations}

\noindent\textbf{Funding.} No funding was received to assist with the preparation of this manuscript.

\smallskip
\noindent\textbf{Competing interests.} The authors have no competing interests to declare that are relevant to the content of this article.

\smallskip
\noindent\textbf{Data availability.} Data sharing is not applicable to this article, as no datasets were generated or analysed during the current study.


\begin{thebibliography}{10}

\bibitem{Babai1980}
L.~Babai, \emph{Finite digraphs with given regular automorphism groups}, Period. Math. Hungar. \textbf{11} (1980), no.~4, 257--270.

\bibitem{Barmak2023}
J.~A. Barmak, \emph{Small posets with prescribed automorphism group}, Period. Math. Hungar. \textbf{86} (2023), no.~1, 210--216.

\bibitem{BarmakBarreto2024}
J.~A. Barmak and A.~N. Barreto, \emph{Smallest posets with given cyclic automorphism group}, Algebr. Comb. \textbf{7} (2024), no.~5, 1307--1318.

\bibitem{BarmakMinian2009}
J.~A. Barmak and E.~G. Minian, \emph{Automorphism groups of finite posets}, Discrete Math. \textbf{309} (2009), no.~10, 3424--3426.

\bibitem{Barreto2021}
A.~N. Barreto, \emph{Sobre los posets m\'as chicos con grupo de automorfismos abeliano dado}, Tesis de Licenciatura, Universidad de Buenos Aires, 2021.

\bibitem{Birkhoff1946}
G.~Birkhoff, \emph{Sobre los grupos de automorfismos}, Rev. Un. Mat. Argentina \textbf{11} (1946), 155--157.

\bibitem{Frucht1939}
R.~Frucht, \emph{Herstellung von Graphen mit vorgegebener abstrakter Gruppe}, Compositio Math. \textbf{6} (1939), 239--250.

\bibitem{Frucht1950}
R.~Frucht, \emph{On the construction of partially ordered systems with a given group of automorphisms}, Amer. J. Math. \textbf{72} (1950), 195--199.

\bibitem{GyenizseHajnalZadori2024}
G.~Gyenizse, P.~Hajnal, and L.~Z\'adori, \emph{Representations of finite groups by posets of small height}, Order \textbf{41} (2024), 593--611.

\bibitem{Thornton1972}
M.~C. Thornton, \emph{Spaces with given homeomorphism groups}, Proc. Amer. Math. Soc. \textbf{33} (1972), 127--131.

\end{thebibliography}
\end{document}